\documentclass{article}

\usepackage[utf8]{inputenc}
\usepackage[small]{my-dgruyter}
\usepackage{microtype}

\usepackage{amsmath,amsfonts,amsthm,amssymb,eucal}
\usepackage{cite}

\newtheorem{lemma}{Lemma}
\newtheorem{theorem}{Theorem}
\newtheorem*{proposition}{Proposition}
\theoremstyle{definition}
\newtheorem{definition}{Definition}
\theoremstyle{remark}
\newtheorem{remark}{Remark}

\begin{document}
	
\author[1]{Moulay Rchid Sidi Ammi}

\author[2]{Delfim F. M. Torres}

\runningauthor{M. R. Sidi Ammi and D. F. M. Torres}

\affil[1]{Group of Mathematical and Numerical Analysis of PDEs and Applications (AMNEA),
Department of Mathematics, Moulay Ismail University,
Faculty of Science and Technology,
B.P. 509, Errachidia, Morocco,
e-mail: sidiammi@ua.pt}

\affil[2]{Center for Research and Development in Mathematics and Applications (CIDMA),
Department of Mathematics, University of Aveiro,
3810-193 Aveiro, Portugal,
e-mail: delfim@ua.pt}

\title{Analysis of fractional integro-differential equations of thermistor type}

\runningtitle{Analysis of fractional integro-differential equations}

\abstract{We survey methods and results of fractional 
differential equations in which an unknown function is under the operation
of integration and/or differentiation of fractional order. 
As an illustrative example, we review results on fractional integral
and differential equations of thermistor type. 
Several nonlocal problems are considered: with Riemann--Liouville,
Caputo, and time-scale fractional operators.
Existence and uniqueness of positive solutions are obtained
through suitable fixed-point theorems in proper Banach spaces.
Additionally, existence and continuation theorems are given,
ensuring global existence.}

\keywords{integral and differential equations;
fractional operators; nonlocal thermistor problems;
time scales; dynamic equations; positive solutions;
local and global existence; fixed point theorems.}

\classification[MSC 2010]{26A33; 26E70; 35A01; 35B09; 45M20; 58J20.}

\maketitle

% -------------------------------------------------

\section{Introduction}

Fractional calculus covers a wide range of classical fields 
in mathematics and its applications, such as Abel's integral 
equation, viscoelasticity, analysis of feedback amplifiers, 
capacitor theory, fractances, electric conductance, 
mathematical biology and optimal control \cite{MR3787674}. 
In particular, Abel integral equations are well studied, 
with many sources devoted to its applications in different fields.
One can say that Abel's integral equations, of first and second kind,
are the most celebrated integral equations of fractional order
\cite{MyID:310}. The former investigations on such equations are due 
to Niels Henrik Abel himself, for the first kind \cite{MR3721889}, 
and to Hille and Tamarkin for the second kind \cite{MR1502959}.
Niels Henrik Abel was led to his equations studying
the \emph{tautochrone} problem (from the Greek prefixes \emph{tauto},
meaning same, and \emph{chrono}, meaning time), that is, 
the problem of determining the shape of a curve for which 
the time taken by an object sliding without friction in uniform gravity 
to its lowest point is independent of its starting point. 
The curve is a cycloid, simultaneously the \emph{tautochrone} 
and the \emph{brachistochrone} curve, which brings together 
the subjects of fractional calculus and the calculus of variations 
\cite{MyID:310,MR3331286}.

Fractional integral equations occur in many situations 
where physical measurements are to be evaluated, e.g., 
in evaluation of spectroscopic measurements 
of cylindrical gas discharges, the study of the solar 
or a planetary atmosphere, the investigation of star densities in
a globular cluster, the inversion of travel times of seismic waves for determination
of terrestrial sub-surface structure, and spherical stereology \cite{ref33,MR2024583}.
For detailed descriptions and analysis of such integral equations of fractional order,
we refer the readers to the books of Gorenflo and Vessella \cite{21}
and Craig and Brown \cite{31}. See also \cite{32}. 
Another field, in which fractional integral equations with 
general weakly singular kernels appear naturally,
is inverse boundary value problems in partial differential equations, 
in particular parabolic ones \cite{GorMin}. Here we are mainly
interested in questions involving existence and uniqueness of solutions
of fractional integral equations.

Existence and uniqueness of solutions for FDEs 
have been intensely studied by many mathematicians
\cite{hilfer,srivastava2,22,29}. However, most available
results have been concerned with existence-uniqueness of solutions for FDEs
on a finite interval. Since continuation theorems for FDEs are not well developed,
results about global existence-uniqueness of solution of FDEs on the half axis
$[0,+\infty )$, by using directly the results from local existence,
have only recently flourished \cite{2,li}. Here we address such issues.
To motivate our study, we can mention two types of electrical circuits, 
which are related with fractional calculus. Circuits of the first type 
consist of capacitors and resistors, which are described by conventional 
(integer-order) models, but for which the circuit itself, as a whole, 
may have non-integer order properties, becoming a fractance device, that is, 
an electrical element that exhibits fractional order impedance properties. 
Circuits of the second type may consist of resistors and capacitors, 
both modeled in the classical sense, and fractances.
In particular, we can consider thermistor-type problems, which are highly
nonlinear and mathematically challenging \cite{MR2397904}.

Inspired by modern developments of thermistors,
where fractional partial differential equations have a crucial role,
we consider here mathematical models and tools that serve as prototypes
for other integral problems. It turns out that available computational methods
for such mathematical problems are not theoretically sound, in the sense
they rely on results of local existence. Here we review the recent theory
of global existence for nonlocal fractional problems of thermistor type
\cite{sidiammi3,sidiammi1,MyID:347,MyID:365}. We begin, in 
Section~\ref{sec:History}, with an historical account of the 
theory of FDEs dealing with Cauchy-type problems and their reduction 
to Volterra integral equations. Section~\ref{section2}
contains some of the main tools in the area: we recall a necessary
and sufficient condition for a subspace of continuous functions
to be precompact; Schauder's fixed point theorem;
and a useful generalization of Gronwall's inequality.
In Section~\ref{sec:RL}, our main concern is existence 
and uniqueness of solution to a fractional order 
nonlocal Riemann--Liouville thermistor problem of the form
\begin{equation}
\label{eq21}
\begin{gathered}
D^{2 \alpha} u(t) = \frac{\lambda f(u(t))}{\left(
\int_{0}^{T} f(u(x))\, dx\right)^{2}}+ h(t) \, , \quad  t \in  (0, T)  \, , \\
I^{\beta} u(t)|_{t=0} = 0 ,  \quad    \forall \beta \in  (0,1],
\end{gathered}
\end{equation}
under suitable conditions on $f$ and $h$ (see Theorem~\ref{thm11}).
We also establish the boundedness of $u$ (Theorem~\ref{thm:bnd}).
Here constant $\lambda$ is a positive dimensionless real parameter.
The unknown function $u$ may be interpreted
as the temperature generated by an electric current flowing
through a conductor \cite{lac2,tza}.
We assume that $T$ is a fixed positive real and $ \alpha > 0$
a parameter describing the order of the FD.
In the case $\alpha =1$ and $h \equiv 0$, \eqref{eq21} becomes
the one-dimensional nonlocal steady state thermistor problem;
the values of $0< \alpha <\frac{1}{2}$ correspond to intermediate processes.
In Section~\ref{sec:Cap}, a more general Caputo thermistor problem
\eqref{13} on the half axis $[0,+\infty )$ is considered, instead
of the bounded interval $[0, T)$ of \eqref{eq21}.
One of the main difficulties lies in handling the nonlocal term
$$
\frac{\lambda f(t, u(t))}{\left(\int_{0}^{t}f(x, u(x))\, dx\right)^{2}}
$$
of problem \eqref{13}, where, in contrast with problem \eqref{eq21},
function $f$ depends on both time and the unknown function $u$.
We are concerned with local existence on a finite interval
(Theorem~\ref{thm3.1}), as well as results of continuation
(Theorem~\ref{thm4.1}) and global existence
(Theorems~\ref{thm5.2} and \ref{thm5.2b})
via Schauder's fixed point theorem.
Finally, in Section~\ref{sec:ts}, we consider fractional
integral and differential equations on time scales, that is,
on arbitrary nonempty closed subsets of the real numbers.
The investigation of dynamic equations on time scales allows
to unify and extend the theories of difference and differential
equations into a single theory \cite{MR1062633}. A time scale 
is a model of time, and the theory has found important applications in several
contexts that require simultaneous modeling of discrete and continuous data. Its
usefulness appears in many different areas, and the reader interested in applications
is refereed to \cite{agarwal1b,agarwal2,bohner1,MR2671876,MR3498485}
and references therein. The idea to join the two subjects of FC
and the calculus on time scales, on a single theory,
was born with the works \cite{MyID:152,MyID:179,MyID:201}
and is under strong current research since 2011: see, e.g.,
\cite{ben,MyID:328,MyID:358,MR3498485,MyID:365}.
Using Schauder's fixed point theorem,
we obtain existence and uniqueness results of positive solutions
for a fractional Riemann--Liouville nonlocal thermistor
problem on arbitrary nonempty closed subsets of the real numbers
(Theorems~\ref{thm12} and \ref{corollary2}).
We end with Section~\ref{sec:conc} of conclusions.

% -------------------------------------------------

\section{Historical account}
\label{sec:History}

Abel's integral equations of first and second kinds 
can be formulated, respectively, as
\begin{equation}
\label{eq:1.1}
f(x)=\int_0^x\frac{k(x,s)g(s)}{(x-s)^\alpha}ds,
\quad 0<\alpha<1, \quad 0 \leq x \leq b,
\end{equation}
and
\begin{equation}
\label{eq:1.1:2nd}
f(x)=\int_0^x\frac{k(x,s)g(s)}{(x-s)^\alpha}ds + g(x),
\quad 0<\alpha<1, \quad 0 \leq x \leq b,
\end{equation}
where $g$ is the unknown function to be found, $f$ is a well behaved function,
and $k$ is the kernel. These celebrated equations 
appear frequently in many physical and engineering problems, like semi-conductors,
heat conduction, metallurgy and chemical reactions \cite{32,21}.
The special case $\alpha = 1/2$ arises often.
If the kernel is given by $k(x,s)=\frac{1}{\Gamma(1-\alpha)}$, 
then \eqref{eq:1.1} is a fractional integral
equation of order $1-\alpha$ \cite{MyID:310}. This problem is a generalization
of the tautochrone problem, and is related with the born of the 
fractional calculus of variations \cite{MyID:404}.
For solving integral equations \eqref{eq:1.1}--\eqref{eq:1.1:2nd} 
of Abel type, several approaches are possible,
e.g., using transformation techniques \cite{r3:5}, 
orthogonal polynomials \cite{r3:6}, integral operators \cite{r3:7}, 
fractional calculus \cite{r3:8,r3:9}, Bessel functions \cite{r3:10}, 
wavelets \cite{r3:12}, methods based on semigroups \cite{r3:13,r3:14}, 
as well as many other techniques \cite{r3:16,r3:17}.

FDEs, in which an unknown function is contained under 
the operation of a FD, have 
a long history, enriched by the intensive development 
of the theory of fractional calculus and their applications 
in the last decades. Fractional ordinary differential equations 
have the following form:
\begin{equation*}
F\left(x, y(x), D^{\alpha_{1}}y(x), D^{\alpha_{2}}y(x), 
\ldots, D^{\alpha_{n}}y(x)\right)=g(x),
\end{equation*}
where $F(x, y_{1}, y_{2}, \ldots, y_{n})$ and $g(x)$ are given functions,  
$D^{\alpha_{k}}$ are fractional differentiation operators of real order
$\alpha_{k} > 0$, $k=1, 2, \ldots, n$ \cite{128}. For example, 
one can consider nonlinear differential equations of the form
\begin{equation}
\label{2}
D^{\alpha}y(x)=f(x, y(x))
\end{equation}
with real $\alpha >0$ or complex $\alpha$, $Re(\alpha) >0$. 
Similarly to the investigation of ordinary differential equations, 
the methods for FDEs are essentially based on the study of equivalent 
Volterra integral equations. Equations \eqref{eq:1.1}--\eqref{eq:1.1:2nd}
are examples of Volterra integral equations of first and second kinds, 
respectively. Several authors have developed methods 
to deal with fractional integro-differential equations 
and construct solutions for ordinary and partial differential equations,
with the general goal of obtaining a unified theory of special functions \cite{78,128}. 
In \cite{GorMin}, analytical solutions to some linear operators of fractional 
integration and fractional differentiation are obtained,
using Laplace transforms. The 1918 note of 
O'Shaughnessy and Post \cite{116},
is one of the first references to develop a method 
for solving the equation
\begin{equation*}
(D^{\frac{1}{2}}y)(x)= \frac{y}{x}
\end{equation*}
with the Riemann--Liouville derivative. Fifteen years later,
in 1933, Fujiwara considered the FDE
\begin{equation*}
(D_{+}^{\alpha}y)(x)=\left(\frac{\alpha}{x}\right)^{\alpha} y(x)
\end{equation*}
using the Hadamard FD of order $\alpha >0$ \cite{49}.
Provided $f(x, y)$ is bounded in a special domain and satisfies a Lipschitz 
condition with respect to $y$, 
Pitcher and Sewell have shown in 1938 how the nonlinear FDE 
\begin{equation}
\label{12}
(D_{a+}^{\alpha}y)(x)=f(x, y(x)), \quad 
0 < \alpha < 1, \quad a \in \mathbb{R},
\end{equation}
in the sense of Riemann--Liouville, can be reduced
to a Volterra integral equation, proving existence and uniqueness 
of a continuous solution to \eqref{12} \cite{117}.
In 1965, Al-Bassam considered the nonlinear Cauchy-type problem 
of fractional order
\begin{equation}
\label{13bis}
\begin{gathered}
(D_{a+}^{\alpha}y)(x)=f(x, y(x)), \quad 0 < \alpha < 1, \quad a \in \mathbb{R},\\
(D_{a+}^{\alpha-1}y)(x)/(x=a) 
= (I_{a+}^{1-\alpha}y)(x)/(x=a) =b_{1},\,\, b_{1} \in \mathbb{R}.
\end{gathered}
\end{equation}
Similarly as before, he applied the method of successive approximations 
to the equivalent reduced Volterra integral equation and the contraction 
mapping method, establishing existence of a unique solution \cite{3}. 
In 1978, Al-Abedeen and Arora considered the problem
\begin{equation*}
(D_{a+}^{\alpha}y)(x)=f(x, y(x)), 
\quad (I_{a+}^{1-\alpha}y)(c)=y_{0},
\quad a<c<b, \quad y_{0} \in  \mathbb{R},
\end{equation*}
with $0 < \alpha \leq 1$, and proved an existence and uniqueness result 
for the corresponding Volterra nonlinear integral equation \cite{al2}.
On the basis of Picard method and Schauder's fixed point theorem, 
Tazali obtained in 1982 two local existence results of 
a continuous solution to \eqref{13bis} \cite{133}.
Interestingly, more general existence and uniqueness results than the ones of \cite{133},
also obtained using a fixed point theorem and equivalent nonlinear integral formulations, 
have been published in 1977 by Leskovski\u\i \cite{83}. 
Similar results, on the basis of fixed-point theorems 
and integral equations, were derived by Semenchuk in 1982 \cite{132}.
In 1988 \cite{37}, El-Sayed examined problem \eqref{2} on a finite interval where the 
FD is considered in the sense of Gelfand and Shilov \cite{51}.
Hadid, in 1995, used a fixed point theorem to prove existence of a solution 
to the corresponding integral equation of \eqref{13bis} \cite{60}.
Using the contraction mapping method on a complete metric space, 
Hayek et al. established in 1999 existence and uniqueness of a continuous solution 
to the Cauchy-type problem described by the following system of FDEs:
\begin{equation*}
(D_{0+}^{\alpha}y)(x)=f(x, y(x)), 
\quad y(a)=b, \quad 0 < \alpha \leq 1, 
\quad a>0, \quad b\in \mathbb{R}^{n}
\end{equation*}
\cite{65}. In 2000, Kilbas, Bonilla and Trujillo studied the Cauchy-type problem
\begin{gather*}
(D_{a+}^{\alpha}y)(x)=f(x, y(x)), 
\quad n-1 < \alpha \leq n, \quad n=-[-\alpha],\\
(D_{a+}^{\alpha-k)}y)(x)/(x=a)  =b_{k},
\quad b_{k} \in \mathbb{R}, \quad k=1, 2, \ldots, n,
\end{gather*}
with complex $\alpha$, $\alpha \in \mathbb{C}$ with $Re(\alpha)>0$, 
on a finite interval $[a, b]$ \cite{69}. By using the method of 
successive expansions and Avery--Henderson and Leggett--Williams 
multiple fixed point theorems on cones, they proved existence of multiple positive solutions 
for the corresponding nonlinear Volterra integral equation \cite{69}. 
See also \cite{Trujillo}. It should be noted, however, that
in some function spaces the proof of equivalence between solutions of Cauchy-type problems 
for FDEs and corresponding reduced Volterra integral equations constitute a major difficulty
\cite{MR3557873}. For more recent results, we refer the reader to \cite{MR2962045}.
The techniques are, however, similar to the ones already mentioned,
being detailed in what follows.

% --------------------------------

\section{Fundamental results}
\label{section2}

Let $C([0,T])$ be the space of all continuous functions on $[0,T]$.
The following three auxiliary lemmas are particularly useful for our purposes.

\begin{lemma}[See \cite{li}]
\label{lem2.2}
Let $M$ be a subset of $C([0,T])$. Then $M$ is precompact
if and only if the following conditions  hold:
\begin{enumerate}
\item $\{u(t):u \in M\}$ is uniformly bounded,
		
\item $\{u(t):u \in M\}$ is equicontinuous on $[0,T]$.
\end{enumerate}
\end{lemma}

\begin{lemma}[Schauder's fixed point theorem \cite{cronin}]
\label{lem2.3 }
Let $U$ be a closed bounded convex subset of a Banach space $X$. If
$T:U\to U$ is completely continuous, then $T$ has a fixed point in $U$.
\end{lemma}

We also recall the following version of Gronwall's lemma.

\begin{lemma}[Generalized Gronwall's inequality \cite{15,35}]
\label{lem5.1}
Let $v:[0,b]\to [0,+\infty )$ be a real function and $w(\cdot)$
be a nonnegative, locally integrable function on $[0,b]$.
Suppose that there exist $a>0$ and $0<\alpha <1$ such that
\begin{equation*}
v(t)\leq w(t)+a\int_0^{t}\frac{v(s)}{(t-s)^{\alpha }}ds.
\end{equation*}
Then there exists a constant $k=k(\alpha )$ such that
\begin{equation*}
v(t)\leq w(t)+ka\int_0^{t}\frac{w(s)}{(t-s)^{\alpha }}ds
\end{equation*}
for $t\in [0,b]$.
\end{lemma}

% -------------------------------------------------

\section{Nonlocal Riemann--Liouville problem}
\label{sec:RL}

Let $0< \alpha <\frac{1}{2}$ and
$X = \left( C([0,T]) , \| \cdot \| \right)$.
For $ x \in C([0,T])$, define the norm
$$
\| x \| =  \sup_{t\in [0,T]} \{ e^{-Nt} |x(t)| \},
$$
which is equivalent to the standard supremum norm for $f \in
C([0,T])$ \cite{sayed}. The use of this norm is technical and
allows us to simplify the integral calculus. By $L^{1}([0,T],
\mathbb{R})$, we denote the set of Lebesgue integrable functions on
$[0,T]$. We consider problem \eqref{eq21} with the following assumptions:
\begin{itemize}

\item[$(H_1)$] \ $f: \mathbb{R}^{+}\rightarrow \mathbb{R}^{+}$ is a Lipschitz continuous
function with Lipschitz constant $L_{f}$ such that $c_{1} \leq f(u) \leq c_{2}$,
with $c_{1}, c_{2}$ two positive constants;

\item[$(H_2)$] \ $h$ is continuous on $(0, T)$ with $h \in L^{\infty}(0, T)$.

\end{itemize}

Our first result asserts existence of a unique solution
to \eqref{eq21} on $C(\mathbb{R}^{+})$ of form
\begin{equation}
\label{eq41}
\begin{split}
u(t) &= I^{2\alpha} \left\{\frac{\lambda f(u)}{\left(
\int_{0}^{T} f(u)\,dx\right)^{2}}+ h(t) \right\}\\
&= \int_0^t\frac{(t-s)^{2 \alpha-1}}{\Gamma(2\alpha)}
\left\{\frac{\lambda f(u)}{\left( \int_{0}^{T} f(u)\, dx\right)^{2}}
+ h(s) \right\}  ds.
\end{split}
\end{equation}

% ----------------------

\subsection{Existence and Uniqueness}

We begin proving equivalence between \eqref{eq21} and \eqref{eq41}
on the space $C(\mathbb{R}^{+})$. This restriction of the space of functions
allows to exclude from the proof a stationary function
with Riemann--Liouville derivative of order $2\alpha$ equal
to $d \cdot t^{2\alpha-1}$, $d\in \mathbb{R}$, which
belongs to the space $C_{1-2\alpha}[0, T]$
of continuous weighted functions.

\begin{lemma}
\label{m:lema1}
Suppose that $\alpha \in (0, \frac{1}{2})$. Then
the nonlocal problem \eqref{eq21} is equivalent to the integral
equation \eqref{eq41} on the space $C(\mathbb{R}^{+})$.
\end{lemma}
\begin{theorem}
\label{thm11}
Let $f$ and $h$ satisfy hypotheses $(H_1)$ and $(H_2)$.
Then there exists a unique solution $u \in X $ of \eqref{eq21} for
all $0 < \lambda < \frac{N^{2 \alpha}}{L_{f}
\left(\frac{1}{\left(c_{1}T\right)^{2}}
+ \frac{2c_{2}^{2}T}{\left(c_{1}T\right)^{4}} e^{NT}\right)}$.
\end{theorem}

\begin{proof}
Let $F : X  \to  X $ be defined by
\[
Fu= I^{2 \alpha} \left\{ \frac{\lambda f(u)}{\left(
\int_{0}^{T} f(u)\, dx\right)^{2}}+ h(t) \right\}.
\]
For a well Chosen  $\lambda >0$  we can prove that the map $ F : X \to X $ is a
contraction and it has a fixed point $u=Fu$. Hence, there exists a
unique $u \in X$ that is the solution to the integral equation
\eqref{eq41}. The result follows from Lemma~\ref{m:lema1}.
\end{proof}

% ----------------------

\subsection{Boundedness}

We now show that the assumption that electrical conductivity $f(u)$ is bounded
(hypothesis $(H_1)$) allows to assert boundedness of $u$.

\begin{theorem}
\label{thm:bnd} Under hypotheses $(H_1)$ and $(H_2)$ and $\lambda
> 0$, if $u$ is the solution of \eqref{eq41}, then
\[
\|u \| \leq  \frac{\left(
\frac{\lambda}{(c_{1}T)^{2}}f(0)+h_{\infty}\right)}{N^{2 \alpha}}
e^{\frac{\lambda L_{f}}{(c_{1}TN^{\alpha})^{2}}}.
\]
\end{theorem}

For more details on the subject see \cite{sidiammi1}.

% -------------------------------------------------

\section{Nonlocal Caputo thermistor problem}
\label{sec:Cap}

Now, our main goal consists to prove global existence of solutions
for a fractional Caputo nonlocal thermistor problem.
Precisely, we consider the following fractional order initial value problem:
\begin{equation}
\label{13}
\begin{gathered}
^{C}D^{2 \alpha}_{0+} u(t) = \frac{\lambda f(t, u(t))}{\left(
\int_{0}^{t}f(x, u(x))\, dx\right)^{2}} \, ,
\quad  t \in  (0, \infty)  \, , \\
u(t)|_{t=0}=u_0,
\end{gathered}
\end{equation}
where $^{C}D^{2 \alpha}_{0+}$ is the fractional Caputo derivative
operator of order $2 \alpha$ with $0 < \alpha < \frac{1}{2}$ a real parameter.
We shall assume the following hypotheses:
\begin{itemize}
\item[$(H_1)$] \ $f: \mathbb{R}^{+} \times \mathbb{R}^{+}\rightarrow \mathbb{R}^{+}$
is a Lipschitz continuous function with Lipschitz constant $L_{f}$ with respect
to the second variable such that $c_{1} \leq f(s,u) \leq c_{2}$
with $c_{1}$ and $c_{2}$ two positive constants;

\item[$(H_2)$] \ there exists a positive constant $M$ such that $f(s, u) \leq M s^{2}$;

\item[$(H_3)$]  \ $|f(s, u)- f(s, v) | \leq  s^{2} |u-v|$ or, in a more general manner,
there exists a constant $\omega \geq 2$ such that $|f(s,u)- f(s,v) | \leq  s^{\omega} |u-v|$.
\end{itemize}

% -------------------------------------------------

\subsection{Local existence theorem}
\label{section33}

In this subsection, a local existence theorem of solutions for \eqref{13}
is obtained by applying Schauder's fixed point theorem. In order
to transform \eqref{13} into a fixed point problem, we give in the
following lemma an equivalent integral form of \eqref{13}.

\begin{lemma}
\label{lem2.1}
Suppose that $(H_1)$--$(H_3)$ holds. Then the
initial value problem \eqref{13} is equivalent to
\begin{equation}
\label{23}
u(t)=u_0+\frac{\lambda}{\Gamma (2\alpha )}\int_0^{t}(t-s)^{2\alpha-1}
\frac{f(s, u(s))}{\left( \int_{0}^{t}f(x, u)\,  dx\right)^{2}}  ds.
\end{equation}
\end{lemma}

\begin{proof}
It is a simple exercise to see that $u$ is a solution of the integral equation \eqref{23}
if and only if it is also a solution of the IVP \eqref{13}.
\end{proof}

\begin{theorem}
\label{thm3.1}
Suppose that conditions $(H_1)$--$(H_3)$ are verified.
Then, \eqref{13} has at least one
solution $u\in C[0,h]$ for some $T\geq h>0$.
\end{theorem}

\begin{proof}
Let
\begin{equation*}
E=\left\{ u\in C[0,T]:\| u-u_0\|_{C[0,T]}
=\sup_{0 \leq t\leq T}| u-u_0| \leq b\right\},
\end{equation*}
where $b$ is a positive constant.
Further, put
\begin{equation*}
D_{h}=\left\{ u:u\in C[0,h],\ \| u-u_0\| _{C[0,h]} \leq b\right\},
\end{equation*}
where
$$
h=\min \left\{ \left (b \left(\frac{ \lambda M }{
\Gamma(2 \alpha+1) c_{1}^{2}}\right)^{-1}\right)^{\frac{1}{2 \alpha}},T\right\}
$$
and $ 0< \alpha < \frac{1}{2}$. It is clear that $h \leq T$. Note also that
$D_{h}$ is a nonempty, bounded, closed, and convex subset of $C[0,h]$. In order
to apply Schauder's fixed point theorem, we define the following operator $A$:
\begin{equation}
\label{33}
(Au)(t)=u_0+\frac{\lambda}{\Gamma (2\alpha )}
\int_0^{t}(t-s)^{2\alpha-1}\frac{f(s, u(s))}{\left(
\int_{0}^{t}f(x, u)\,  dx\right)^{2}}  ds,\quad t\in [0,h].
\end{equation}
It is clear that all solutions of \eqref{13} are fixed points of \eqref{33}.
Then, by assumptions $(H_1)$ and $(H_2)$, it follows that $AD_{h}\subset D_{h}$.
Our next step, in order to prove Theorem~\ref{thm3.1},
consists to use the following two technical lemmas.
\begin{lemma}
\label{lemma:in1}
The operator $A$ is continuous.
\end{lemma}
\begin{lemma}
\label{lemma:in2}
The operator $AD_{h}$ is continuous.
\end{lemma}
One can prove that  $\{(Au)(t):u\in D_{h}\}$ is equicontinuous.
Taking into account that $AD_{h}\subset D_{h}$, we infer that
$AD_{h}$ is precompact. This implies that $A$ is completely continuous.
As a consequence of Schauder's fixed point theorem and Lemma~\ref{lem2.1},
we conclude that problem \eqref{13} has a local solution.
This ends the proof of Theorem~\ref{thm3.1}.
\end{proof}

% -------------------------------------------------

\subsection{Continuation results}
\label{section43}

Now we give a continuation theorem for the fractional
Caputo nonlocal thermistor problem \eqref{13}.
First, we present the definition of noncontinuable solution.

\begin{definition}[See \cite{40}]
\label{def4.1}
Let $u(t)$ on $(0,\beta )$ and $\tilde{u}(t)$ on $(0,\tilde{\beta})$
be both solutions of \eqref{13}. If $\beta <\tilde{\beta}$ and
$u(t)=\tilde{u}(t)$ for $t\in (0,\beta )$, then we say
that $\tilde{u}(t)$ can be continued to $(0,\tilde{\beta})$.
A solution $u(t)$ is noncontinuable if it has no continuation.
The existing interval of the noncontinuable solution $u(t)$
is called the maximum existing interval of $u(t)$.
\end{definition}

\begin{theorem}
\label{thm4.1}
Assume that conditions $(H_1)$--$(H_3)$ are satisfied. Then
$u=u(t)$, $t\in (0,\beta )$, is noncontinuable if and if only
for some $\eta \in \left(0,\frac{\beta }{2}\right)$ and any bounded closed subset
$S\subset [\eta ,+\infty )\times\mathbb{R}$ there exists a
$t^{\ast }\in [ \eta ,\beta )$ such that $(t^{\ast},u(t^{\ast }))\notin S$.
\end{theorem}

\begin{proof}
Suppose that there exists a compact subset
$S\subset [ \eta,+\infty )\times\mathbb{R}$ such that
$$
\left\{ (t,u(t)):t\in [ \eta ,\beta )\right\} \subset S.
$$
Compactness of $S$ implies $\beta <+\infty $.
The proof follows from Lemmas~\ref{insTh2:lem1} and \ref{insTh2:lem2}.
\end{proof}

\begin{lemma}
\label{insTh2:lem1}
The limit  $\lim_{t\to \beta ^{-}}u(t)$ exists.
\end{lemma}
\begin{proof}
The proof is based on the Cauchy's convergence criterion.
\end{proof}

The second step of the proof of Theorem~\ref{thm4.1}
consists to show the following result.
\begin{lemma}
\label{insTh2:lem2}
Function $u(t)$ is continuable.
\end{lemma}
\begin{proof}
As $S$ is a closed subset, we can say that $(\beta ,u^{\ast })\in S$.
Define $u(\beta )=u^{\ast }$. Hence, $u(t)\in C[0,\beta ]$.
Then we define the operator $K$ by
\begin{equation*}
(Kv)(t)=u_1+\frac{\lambda}{\Gamma (2 \alpha )}\int_{\beta }^{t}(t-s)^{2
\alpha-1} \frac{ f(s, v(s))}{\left( \int_{0}^{t}f(x, v)\, dx\right)^{2}} ds,
\end{equation*}
where
\begin{equation*}
u_1=u_0+\frac{\lambda}{\Gamma (2\alpha )}\int_0^{\beta }(t-s)^{2\alpha-1}
\frac{ f(s, v(s))}{\left( \int_{0}^{t}f(x, v)\, d x\right)^{2}} ds,
\end{equation*}
$v \in C([\beta ,\beta +1])$, $t\in [ \beta ,\beta+1]$. Set
\begin{equation*}
E_{b}=\left\{ (t,v):\beta \leq t\leq \beta +1,| v|
\leq \max_{\beta \leq t\leq \beta +1}| u_1(t)| +b\right\}
\end{equation*}
and
\begin{equation*}
E_{h}=\left\{ v\in C[\beta ,\beta +1]:\max_{t\in [ \beta ,\beta+h]}| v(t)-u_1(t)|
\leq b,v(\beta )=u_1(\beta)\right\},
\end{equation*}
where $h=\min\left\{ \left (b \left(\frac{ \lambda M}{\Gamma(2 \alpha +1 )
c_{1}^{2}} \right )^{-1}\right)^{\frac{1}{2 \alpha}},1\right\}$.
Similarly to the proof of Theorem~\ref{thm3.1}, we show that $K$
is completely continuous on $E_{b}$, which yields that operator $K$ 
is continuous. We show that $KE_{h}$ is equicontinuous.
Consequently, $K$ is completely continuous. Then, Schauder's fixed
point theorem can be applied to obtain that operator $K$ has a fixed point
$\tilde{u}(t)\in E_{h}$. On other words, we have
\begin{equation*}
\begin{aligned}
\tilde{u}(t)
&= u_1+\frac{\lambda}{\Gamma (2\alpha )}\int_{\beta}^{t}
(t-s)^{2\alpha -1} \frac{ f(s, \tilde{u}(s))}{\left(
\int_{0}^{t}f(x, \tilde{u}(x))\, dx\right)^{2}} ds\\
&= u_0+\frac{\lambda}{\Gamma (2\alpha )}
\int_0^{t}(t-s)^{2\alpha -1} \frac{ f(s, \tilde{u}(s))}{\left(
\int_{0}^{t}f(x, \tilde{u}(x))\, dx\right)^{2}} ds,
\end{aligned}
\end{equation*}
$t\in [ \beta ,\beta +h]$, where
\[
\tilde{u}(t)=\begin{cases}
u(t), & t\in (0,\beta ] \\
\tilde{u}(t), & t\in [ \beta ,\beta +h].
\end{cases}
\]
It follows that $\tilde{u}(t)\in C[0,\beta +h]$ and
\begin{equation*}
\tilde{u}(t)=u_0+\frac{\lambda}{\Gamma (2\alpha )}
\int_0^{t}(t-s)^{2\alpha-1}\frac{ f(s, \tilde{u}(s))}{\left(
\int_{0}^{t}f(x, \tilde{u}(x))\, dx\right)^{2}} ds.
\end{equation*}
Therefore, according to Lemma~\ref{lem2.1}, $\tilde{u}(t)$ is a solution of
\eqref{13} on $(0,\beta +h]$. This is absurd because $u(t)$ is
noncontinuable. This completes the proof of Lemma~\ref{insTh2:lem2}.
\end{proof}

Similarly to Theorem~\ref{thm11},
uniqueness of solution to problem \eqref{13}
is derived from the proof of Theorem~\ref{thm4.1}
for a well chosen $\lambda$.

% -------------------------------------------------

\subsection{Global existence of solutions}
\label{section53}

Now we provide two sets of sufficient conditions
for the existence of a global solution for \eqref{13}
(Theorems~\ref{thm5.2} and \ref{thm5.2b}).
We begin with an auxiliary lemma.

\begin{lemma}
\label{thm5.1}
Suppose that conditions $(H_1)$--$(H_3)$ hold. Let $u(t)$ be a solution
of \eqref{13} on $(0,\beta )$. If $u(t)$ is bounded on $[\tau ,\beta )$
for some $\tau >0$, then $\beta =+\infty$.
\end{lemma}

\begin{proof}
Follows immediately from the results of Subsection~\ref{section43}.
\end{proof}

\begin{theorem}
\label{thm5.2}
Suppose that conditions $(H_1)$--$(H_3)$ hold.
Then \eqref{13} has a solution in $C([0,+\infty ))$.
\end{theorem}

\begin{proof}
The existence of a local solution $u(t)$ of \eqref{13} is ensured thanks to
Theorem~\ref{thm3.1}. We already know, by Lemma~\ref{lem2.1},
that $u(t)$ is a also a solution to the integral equation
\begin{equation*}
u(t)=u_0+\frac{\lambda}{\Gamma (2\alpha )}\int_0^{t}(t-s)^{2\alpha
-1}\frac{ f(s, u(s))}{\left( \int_{0}^{t}f(x, u(x))\, dx\right)^{2}}ds.
\end{equation*}
Suppose that the existing interval of the noncontinuable solution $u(t)$ is
$(0,\beta)$, $\beta<+\infty$. Then,
\begin{align*}
|u(t)|
&= \left| u_0+\frac{\lambda}{\Gamma (2\alpha)}
\int_0^{t}(t-s)^{2\alpha -1}\frac{ f(s, u(s))}{\left(
\int_{0}^{t}f(x, u(x))\, dx\right)^{2}}ds\right| \\
& \leq |u_0|+ \frac{\lambda}{\Gamma (2\alpha )} \frac{1}{(c_{1}t)^{2}}
\int_0^{t}(t-s)^{2\alpha -1}\left| f(s, u(s))\right|ds\\
&\leq |u_0| + \frac{\lambda }{\Gamma (2\alpha )}
\frac{1}{c_{1}^{2}} \int_0^{t} \frac{|u(s)|}{(t-s)^{1- 2\alpha}} ds.
\end{align*}
By Lemma~\ref{lem2.1}, there exists a constant $k(\alpha)$ such that,
for $t\in (0, \beta)$, we have
\begin{align*}
| u(t)| & \leq |u_0| + k |u_0|\frac{\lambda }{\Gamma (2\alpha )}
\frac{1}{c_{1}^{2}} \int_0^{t} (t-s)^{2\alpha-1} ds,
\end{align*}
which is bounded on $(0, \beta)$. Thus, by Lemma~\ref{thm5.1},
problem \eqref{13} has a solution $u(t)$ on $(0,+\infty )$.
\end{proof}

Next we give another sufficient condition ensuring
global existence for \eqref{13}.

\begin{theorem}
\label{thm5.2b}
Suppose that there exist positive constants
$c_{3}$, $c_{4}$ and  $c_{5}$ such that
$c_{3} \leq |f(s, x)|\leq c_{4}|x|+ c_{5}$.
Then \eqref{13} has a solution in $C([0,+\infty ))$.
\end{theorem}

\begin{proof}
Suppose that the maximum existing interval of $u(t)$ is
$(0,\beta )$, $\beta<+\infty$. We claim that $u(t)$
is bounded on $[\tau ,\beta )$ for any $\tau \in (0,\beta )$.
Indeed, we have
\begin{align*}
|u(t)|
&= \left| u_0+\frac{\lambda}{\Gamma (2\alpha )}
\int_0^{t}(t-s)^{2\alpha -1}\frac{ f(s, u(s))}{\left(
\int_{0}^{t}f(x, u(x))\, dx\right)^{2}}ds\right| \\
&\leq |u_0|+\frac{\lambda}{\Gamma (2\alpha )}
\frac{c_{3}}{(c_{1}\tau)^{2}}\int_0^{t}(t-s)^{2\alpha -1} ds
+ \frac{\lambda }{\Gamma (2\alpha )} \frac{c_{2}}{(c_{1}\tau)^{2}}
\int_0^{t}  \frac{|u(s)|}{(t-s)^{1- 2\alpha}} ds.
\end{align*}
If we take
$$
w(t)=|u_0|+ \frac{\lambda}{\Gamma (2\alpha )}
\frac{c_{3}}{(c_{1}\tau)^{2}}\int_0^{t}(t-s)^{2\alpha -1}ds,
$$
which is bounded, and
$$
a=\frac{\lambda c_{2}}{\Gamma (2\alpha )} \frac{1}{(c_{1}\beta)^{2}},
$$
it follows,  according with Lemma~\ref{lem5.1}, that
$v(t)=|u(t)|$ is bounded. Thus, by Lemma~\ref{thm5.1},
problem \eqref{13} has a solution $u(t)$ on $(0,+\infty )$.
\end{proof}

For more details on the subject, see
\cite{sidiammi3,MR3736617,sidiammi1,sidiammi2,GL}.

% -------------------------------------------------

\section{Fractional problems on arbitrary time scales}
\label{sec:ts}

Throughout the reminder of this chapter,
we denote by $\mathbb{T}$ a time scale,
which is a nonempty closed subset of $\mathbb{R}$ with its inherited topology.
For convenience, we make the blanket assumption that $t_{0}$ and $T$
are points in $\mathbb{T}$. Our main concern is to prove existence and uniqueness
of solution to a fractional order nonlocal thermistor problem of the form
\begin{equation}
\label{eq12}
\begin{gathered}
^{\mathbb{T}}D^{2 \alpha}_{t_{0+}} u(t)
= \frac{\lambda f(u(t))}{\left(\int_{t_{0}}^{T} f(u(x))\, \triangle x\right)^{2}}
\, , \quad  t \in (t_{0}, T)  \, , \\
^{\mathbb{T}}I_{t_{0+}}^{\beta}u(t_{0})=0,
\quad    \forall \, \beta \in  (0, 1),
\end{gathered}
\end{equation}
under suitable conditions on $f$ as described below.
We assume that $\alpha \in (0,1)$ is a parameter describing
the order of the FD;
$^{\mathbb{T}}D^{2 \alpha}_{t_{0+}}$ is
the left Riemann--Liouville FD operator of order
$2 \alpha$ on $\mathbb{T}$; $^{\mathbb{T}}I^{\beta}_{t_{0+}}$ is
the left Riemann--Liouville FI operator
of order $\beta$ defined on $\mathbb{T}$ by \cite{MyID:328}
(see Section~\ref{section22}, where these definitions
and main properties of the fractional operators
on time scales are recalled). As before, $u$ may be interpreted
as the temperature inside the conductor
and $f(u)$ the electrical conductivity of the material.

In the literature, many existence results for dynamic equations on time scales
are available \cite{dogan1,dogan2}. In recent years, there has been also
significant interest in the use of FDEs
in mathematical modeling \cite{MR3370824,MR3529931,MR3535730}.
However, much of the work published to date has been
concerned separately, either by the time-scale community or
by the fractional one. Results on fractional dynamic equations
on time scales are scarce \cite{ahmad}. Here we give existence
and uniqueness results for the fractional order
nonlocal thermistor problem on time scales \eqref{eq12},
putting together time scale and fractional domains.
According with \cite{MR3323914,MyID:358,MR3498485},
this is quite appropriate from the point
of view of practical applications.
Our main aim is to prove existence of solutions for \eqref{eq12}
using a fixed point theorem and, consequently, uniqueness
(see Theorems~\ref{thm12} and \ref{corollary2}).
For more details see \cite{MyID:365}.

% -------------------------------------------------

\subsection{Fractional calculus on time scales}
\label{section22}

We deal with the notions of Riemann--Liouville FIs
and FDs on time scales, the so called BHT
fractional calculus on time scales \cite{MyID:358}.
For local approaches, we refer the reader to \cite{ben,MyID:320}.
Here we are interested in nonlocal operators, which are the ones who make
sense with respect to thermistor-type problems \cite{MR2397904,MR2980253}.
Although we restrict ourselves to the delta approach on time scales,
similar results are trivially obtained for the nabla fractional
case \cite{MyID:224}.

\begin{definition}[Riemann--Liouville FI on time scales \cite{MyID:328}]
\label{def12}
Let $\mathbb{T}$ be a time scale and $[a, b]$ an interval of $\mathbb{T}$.
Then the left fractional integral on time scales of order $ 0 < \alpha <1$
of a function $g: \mathbb{T} \rightarrow \mathbb{R}$ is defined by
$$
^{\mathbb{T}}I^{\alpha}_{a+} g(t)
= \int_a^t \frac{(t - s)^{ \alpha-1}}{\Gamma( \alpha)} g(s)\, \triangle s,
$$
where $\Gamma $ is the Euler gamma function.
\end{definition}

The left Riemann--Liouville FD operator
of order $\alpha$ on time scales is then defined using
Definition~\ref{def12} of FI.

\begin{definition}[Riemann--Liouville FD on time scales \cite{MyID:328}]
\label{def22}
Let $\mathbb{T}$ be a time scale, $[a, b]$ an interval of $\mathbb{T}$,
and $\alpha \in (0,1)$. Then the left Riemann--Liouville FD
on time scales of order $\alpha$ of a function $g: \mathbb{T} \rightarrow \mathbb{R}$
is defined by
$$
^{\mathbb{T}}D^{\alpha}_{a+} g(t)
= \left (\int_a^t \frac{(t - s)^{ -\alpha}}{\Gamma(1- \alpha)}
g(s)\, \triangle s \right)^{\triangle}.
$$
\end{definition}

\begin{remark}
If $\mathbb{T}= \mathbb{R}$, then we obtain from Definitions~\ref{def12} and \ref{def22},
respectively, the usual left Rieman--Liouville FI and FD.
\end{remark}

\begin{proposition}[See \cite{MyID:328}]
Let $\mathbb{T}$ be a time scale,
$g: \mathbb{T} \rightarrow \mathbb{R}$ and $0< \alpha < 1$. Then
\[
{^{\mathbb{T}}D}^{\alpha}_{a+} g
= \triangle \circ {^{\mathbb{T}}I}^{1- \alpha}_{a+} g.
\]
\end{proposition}

\begin{proposition}[See \cite{MyID:328}]
If $\alpha >0$ and $g \in C([a, b])$, then
\[
^{\mathbb{T}}D^{\alpha}_{a+}
\circ {^{\mathbb{T}}I}^{\alpha}_{a+}g= g.
\]
\end{proposition}

\begin{proposition}[See \cite{MyID:328}]
\label{prop:I:D}
Let $g \in C([a, b])$, $0 <\alpha < 1$. If
${^{\mathbb{T}}I}^{1-\alpha}_{a+} u(a)=0$, then
\[
^{\mathbb{T}}I^{\alpha}_{a+}
\circ {^{\mathbb{T}}D}^{\alpha}_{a+}g= g.
\]
\end{proposition}

\begin{theorem}[See \cite{MyID:328}]
Let $g \in C([a, b])$, $\alpha > 0$, and $_{a}^{\mathbb{T}}I^{\alpha}_{t}([a, b])$
be the space of functions that can be represented by the Riemann--Liouville
$\triangle$-integral of order $\alpha$ of some $C([a, b])$-function. Then,
\[
g \in\, _{a}^{\mathbb{T}}I^{\alpha}_{t}([a, b])
\]
if and only if
\[
^{\mathbb{T}}I^{1-\alpha}_{a+} g \in C^{1}([a, b])
\]
and
\[
^{\mathbb{T}}I^{1-\alpha}_{a+} g(a) =0.
\]
\end{theorem}

The following result of the calculus on time scales is also useful.

\begin{proposition}[See \cite{ahmad}]
\label{proposition252}
Let $\mathbb{T}$ be a time scale and $g$ an increasing continuous
function on the time-scale interval $[a, b]$. If $G$ is the extension
of $g$ to the real interval $[a, b]$ defined by
$$
G(s):=
\begin{cases}
g(s) & \mbox{ if } s \in \mathbb{T},\\
g(t) &  \mbox{ if } s \in (t, \sigma(t)) \notin  \mathbb{T},
\end{cases}
$$
then
\[
\int_{a}^{b} g(t) \triangle t \leq  \int_{a}^{b} G(t)  dt,
\]
where $\sigma: \mathbb{T} \rightarrow \mathbb{T}$ is the forward
jump operator of $\mathbb{T}$ defined by
$$
\sigma(t):= \inf \{ s \in \mathbb{T}: s > t\}.
$$
\end{proposition}

% -----------------------------------------

\subsection{Existence}

We begin by giving an integral representation to our problem \eqref{eq12}.
Note that the case $0<\alpha<\frac{1}{2}$ is coherent
with our fractional operators with $2 \alpha -1 < 0$.

\begin{lemma}
Let $ 0< \alpha < \frac{1}{2}$. Problem
\eqref{eq12} is equivalent to
\begin{equation}
\label{eq22}
u(t) = \frac{\lambda}{\Gamma (2 \alpha )} \int_{t_{0}}^{t} (t-s)^{2 \alpha -1}
\frac{f(u(s))}{\left(\int_{t_{0}}^{T} f(u)\, \triangle x\right)^{2}} \triangle s.
\end{equation}
\end{lemma}

\begin{proof}
We have
\begin{equation*}
\begin{split}
^{\mathbb{T}}D^{2 \alpha}_{t_{0+}} u(t)
&=  \frac{\lambda}{\Gamma ( 2 \alpha )} \left(
\int_{t_{0}}^{t} (t-s)^{2 \alpha -1} \frac{f(u(s))}{\left(
\int_{t_{0}}^{T} f(u)\, \triangle x\right)^{2}}\triangle s\right )^{\triangle}\\
&= \left( ^{\mathbb{T}}I^{1- 2 \alpha}_{t_{0+}} u(t) \right )^{\triangle}
= \left(\triangle \, \circ \, ^{\mathbb{T}}I^{1- 2 \alpha}_{t_{0+}} \right) u(t).
\end{split}
\end{equation*}
The result follows from Proposition~\ref{prop:I:D}:
$^{\mathbb{T}}I^{2 \alpha}_{t_{0+}}  \circ
\left ( ^{\mathbb{T}}D^{ 2 \alpha}_{t_{0+}} (u) \right) = u$.
\end{proof}

For the sake of simplicity, we take $t_{0}=0$.
It is easy to see that \eqref{eq12} has a solution
$u=u(t)$ if and only if $u$ is a fixed point of the operator
$K: X \rightarrow X $ defined by
\begin{equation}
\label{eq32}
Ku(t)=  \frac{\lambda}{\Gamma ( 2 \alpha )} \int_{0}^{t} (t-s)^{2 \alpha -1}
\frac{f(u(s))}{\left(\int_{0}^{T} f(u)\, \triangle x\right)^{2}} \triangle s.
\end{equation}

To prove existence of solution,
we begin by showing that the operator $K$ defined
by \eqref{eq32} verifies the conditions of Schauder's
fixed point theorem.

\begin{lemma}
\label{m:lema2}
The operator $K$ is continuous.
\end{lemma}

\begin{lemma}
The operator $K$ sends bounded sets into bounded sets on
$\mathbb{C}([0, T], \mathbb{R})$.
\end{lemma}
\begin{lemma}
Operator $K$ sends bounded sets into equicontinuous
sets of $\mathbb{C}(I, \mathbb{R})$.
\end{lemma}
It follows by Schauder's fixed point theorem
that \eqref{eq12} has a solution on $I$. 
We have just proved Theorem~\ref{thm12}.

\begin{theorem}[Existence of solution]
\label{thm12}
Let $ 0< \alpha < \frac{1}{2}$
and $f$ satisfies hypothesis $(H_1)$.
Then there exists a  solution $u \in X $ of
\eqref{eq12} for all $\lambda > 0$.
\end{theorem}

% ----------------------

\subsection{Uniqueness}

We now derive uniqueness of solution to problem \eqref{eq12}.

\begin{theorem}[Uniqueness of solution]
\label{corollary2}
Let $ 0< \alpha < \frac{1}{2}$ and $f$ be a function
satisfying the hypothesis $(H_1)$. If
$$
0 < \lambda <  \left(
\frac{ T^{2 \alpha}L_{f}}{(c_{1}T)^{2}\Gamma (2 \alpha +1 )}
+ \frac{2 c_{2}^{2}T^{2(\alpha +1)}
L_{f}}{(c_{1}T)^{4}\Gamma (2 \alpha+1 )} \right)^{-1},
$$
then the solution predicted by Theorem~\ref{thm12} is unique.
\end{theorem}

The map $ K : X \to X $ is a contraction. It follows by the Banach principle
that it has a fixed point $u=Fu$. Hence, there exists a unique
$u \in X$ solution of \eqref{eq22}.

% -------------------------------------------------

\section{Conclusion}
\label{sec:conc}

We gave a survey on methods and results 
based on the reduction of FDEs to integral equations. 
As an example, we have reviewed the main results and proof techniques
of \cite{MR3736617,sidiammi1,MyID:365}.
The employed mathematical techniques are quite general and effective,
and can be used to cover a wide class of integral equations of fractional order.
We trust the theoretical results here reported will have a positive impact
on the development of computer methods for
integral equations of fractional order.

% -------------------------------------------------

\begin{acknowledgement}
the authors are grateful to the R\&D
unit UID/MAT/04106/2013 (CIDMA)
and to two anonymous referees
for constructive comments.
\end{acknowledgement}

% -------------------------------------------------

% -------------------------------------------------

\end{document}